\documentclass[preprint,authoryear,12pt]{article}
\usepackage{newlfont}
\usepackage{a4wide}
\usepackage{amsmath}
\usepackage{inputenc}
\usepackage{amsfonts}
\usepackage{amssymb}
\usepackage{amsthm}
\usepackage{newlfont}
\usepackage{graphicx}
\usepackage{subfigure}
\usepackage{mathrsfs}
\usepackage[authoryear]{natbib}
\usepackage{color}

\usepackage{comment}

\usepackage[running]{lineno}

\newtheorem{theorem}{Theorem}

\newtheorem{proposition}{Proposition}

\newtheorem{definition}{Definition\rm}
\newtheorem{remark}{Remark}

\theoremstyle{definition}

\newcommand{\PP}{{\mathcal P}}
\renewcommand{\AA}{{\mathcal A}}

\newcommand{\IR}{{\mathrm{I\!R}}}

\newcommand{\IN}{{\mathrm{I\!N}}}

\newcommand{\HH}{{\mathcal H}}

\renewcommand{\SS}{{\mathscr S}}
\newcommand{\CC}{{\mathscr C}}

\newcommand{\cob}{\overline{\mathop{\mathrm{co}}}}

\begin{document}


\title{The Dunford-Choquet integral\footnote{Last updated: November 6, 2019}}
\author{Y. Askoura}

\maketitle

{LEMMA, Universit\'e Paris II, 
4 rue Blaise Desgoffe, 75006 Paris.}

{youcef.askoura@u-paris2.fr}

\begin{abstract} We define and investigate some preliminary properties of an integral of vector valued functions and its canonically associated indefinite capacity. More precisely, we adapt the Dunford integral to the Choquet integration setting for Banach-valued functions. 
\end{abstract}




\section{Introduction}

We try, in this note, to provide meaning for the integral of a Banach valued function relatively to a capacity. We follow \citet{CHO55,DEL71}, and many others, integrating real valued functions relatively to capacities and establish a  Dunford type integral for vector valued functions. The intuition is somewhat disturbed by the fact that an integral with respect to a capacity is not generally additive as long as the underlying capacity is not. Thus, the integral of a vector valued function with respect to a capacity cannot be defined canonically as an element of a bidual, like the  Dunford integral in case of measures. Hence, given a measurable space $(\Omega,\AA)$, we set simply, under some additional conditions making the entities meaningful, the  Dunford-Choquet integral of a function $f$ defined from $(\Omega,\AA)$, endowed with a capacity $\nu$, into a Banach space $X$ on every $E\in \AA$ as a function~: 
$$
\begin{array}{rl}
I_f(E) :&X^*\rightarrow \IR\\
&x^*\mapsto I_f(E)x^*=\int_\Omega x^* f \chi_E d\nu,
\end{array}
$$
where, $X^*$ is the dual of $X$ and $\chi_E$ is the characteristic function of $E$. We establish some properties showing that $I_f$ behave well and as an indefinite integral, it defines a capacity from $\AA$ to the set of positively homogeneous functions on $X^*$ inheriting some properties from the considered capacity $\nu$.   

\section{Integration of vector valued functions}
Let $\Omega$ be an arbitrary set and $\SS$ a collection of subsets of $\Omega$ stable under union and intersection. A capacity on $(\Omega,\SS)$ is a set function $\nu : \SS\rightarrow \IR_+$ satisfying $\nu(\emptyset)=0$. It is said to be 
\begin{itemize}
\item monotonic : iff for all $A,B\in \SS$, with $A\subset B$, $\nu(A)\leq \nu(B)$.
\item Sub-additive : iff for all $A,B\in \SS$, $A\cap B=\emptyset$,  $\nu(A\cup B)\leq \nu(A)+\nu(B)$.
\item Sub-modular : iff for all $A,B\in \SS$,  $\nu(A\cup B)+\nu(A\cap B)\leq \nu(A)+\nu(B)$.
\item Super-additive : iff for all $A,B\in \SS$, $A\cap B=\emptyset$,  $\nu(A\cup B)\geq \nu(A)+\nu(B)$.
\item Super-modular : iff for all $A,B\in \SS$,  $\nu(A\cup B)+\nu(A\cap B)\geq \nu(A)+\nu(B)$.
\item Continuous from below : iff for every sequence $(A_n)_{n\in \IN}$ of subsets of $\SS$ such that $A_n\subset A_{n+1}$ for every $n\in \IN$ (increasing) and $\cup_{n\in \IN} A_n=A\in \SS$, $\underset{n \rightarrow +\infty}{\lim} \nu(A_n)=\nu(A)$. 
\item Continuous from above : iff for every sequence $(A_n)_{n\in \IN}$ of subsets of $\SS$ such that $A_n\supset A_{n+1}$ for every $n\in \IN$ (decreasing) and $\cap_{n\in \IN} A_n=A\in \SS$, $\underset{n \rightarrow +\infty}{\lim} \nu(A_n)=\nu(A)$. 

\item Continuous : iff it is both continuous from above and from below.
\end{itemize}

For a sequence $(A_n)_{n\in \IN}$ of subsets of $\SS$ such that $A_n\subset A_{n+1}$ for every $n\in \IN$ and $\cup_{n\in \IN} A_n=A$, we write in the sequel, $A_n\uparrow A$, and for a sequence $(A_n)_{n\in \IN}$ of subsets of $\SS$ such that $A_n\supset A_{n+1}$ for every $n\in \IN$ and $\cap_{n\in \IN} A_n=A$, we write  $A_n\downarrow A$.

\medskip
From now on consider a measurable space $(\Omega,\AA)$, where $\AA$ is a $\sigma-$algebra, and let $\nu$ be  a finite monotonic capacity defined on $\AA$, that is $\nu:\AA\rightarrow \IR_{+}, \nu(\emptyset)=0$ and $\nu(A)\leq \nu(B)$ whenever $A,B\in \AA$ and $A\subset B$. By \emph{the Choquet integral} of a measurable function $h :\Omega\rightarrow \IR$, we mean the usual asymmetric Choquet integral of $h$~: 

$$ \int_\Omega hd\nu=\int_0^{+\infty} \nu(\{\omega\in \Omega : h(\omega)>t\}) d\lambda(t) +\int^0_{-\infty}\left[ \nu(\{\omega\in \Omega : h(\omega)>t\}) -\nu(\Omega)\right]d\lambda(t).$$

where $\lambda$ is the Lebesgue measure on $\IR$. Note that since the functions within the integral symbole are sign constant, the improper Riemann and Lebesgue integrals coincide. The function $h$ will be said, below, Choquet integrable or simply $\nu$-integrable, if all the integrals above are finite.

\medskip
Let $X$ be a Banach space. Denote by $X^*$ the dual of $X$ and by $X^{**}$ its bidual. Let $E\in \AA$ and $f:\Omega\rightarrow X$. Then, $f$ is said to be weakly measurable iff $x^*f$ is measurable, for every $x^*\in X^*$.  The function $f$ is said to be weakly integrable on $E$, iff it is weakly measurable and $\int_\Omega x^*f \chi_Ed\nu$ exists for every $x^*\in X^*$. It is said weakly integrable iff it is weakly integrable on every $E\in \AA$. We state, in a canonical way, the following~:

\begin{definition}\label{DDC} Let $E\in \AA$  and $f:\Omega \rightarrow X$ be a weakly integrable function. The Dunford-Choquet integral of $f$ on $E$ is the function  $I_f(E) :X^*\rightarrow \IR$ defined at every $x^*\in X^*$ by 

$$ I_f(E)x^*= \int_\Omega x^* f\chi_E d\nu$$
A weakly integrable function $f:\Omega\rightarrow X$ is said to be Dunford-Choquet integrable iff $I_f(E) $ exists for every $E\in \AA$.
\end{definition}

Let us make some immediate remarks. 

\begin{itemize} 
\item[1)] If $\nu$ is a measure the function $I_f(E)$ is identical to the classical Dunford integral of $f$ on $E$. In this case,  it is known that $I_f(E)\in X^{**}$.

\item[2)] As defined above, it is not meaningful to define analogously the generalized Pettis integral for capacities. Indeed, if $\nu$ is not additive, the integral with respect to $\nu$ fails, generally, to be linear. Hence, $I_f(E)$ has not to belong to $X^{**}$ or to $X$ as it is the case for measures. 

\item[3)] Consider a particular case of a super-modular monotonic capacity $\nu$. In this case \citep{SCH86,DEN94}, there is a set of measures $K$ such that the integral on $E\in \AA$ with respect to $\nu$ of a real valued function $f :\Omega\rightarrow \IR$ can be obtained by $$ \int_\Omega f \chi_E d\nu=\underset{p\in K}{\inf} \int_\Omega f\chi_E dp=\underset{p\in K}{\inf} \int_E f dp.$$ 
Here, we set $X=X^*=X^{**}=\IR$. Clearly, for every $p\in K$, $\int_E f d p$ defines an element $x^{**}_E\in X^{**}$ by $ x^{**}_E(x^*)=\int_Ex^* f d p$. Hence, $I_f(E)(\alpha)=min \{ \int_E \alpha f d p : p\in K\}$. For instance, if $K=\{p,q\}$, $\int_E f d p=1,\int_E f d q=2$, $I_f(E)$ is a function from $\IR$ into $\IR$ defined by $I_f(E)(\alpha)=\alpha$ for $\alpha>0$ and $I_f(E)(\alpha)=2\alpha$ if $\alpha\leq 0$.

\item[4)] It results straightforwardly, from the definition, that $I_f(E)$ is positively homogenous like the Choquet integral. 
\end{itemize}

\begin{theorem} \label{DCBT} Let $E\in \AA$. Assume that $\nu$ is sub-additive and continuous from above. Let $f:\Omega\rightarrow X$ be weakly measurable and satisfies~:

\begin{equation}\label{DCB}
\int_\Omega \|f\|\chi_E d\nu <+\infty
\end{equation}
Then,  $I_{f}(E)$ is continuous for the weak* topology. 

\end{theorem}

\begin{proof}

Note first that if $\nu$ is sub-additive and continuous from above, then it is also continuous from below and then continuous. 
Indeed, let a sequence $(B_n)_{n\in \IN}$ in $\AA$ and $B\in \AA$ such that $B_n\uparrow B$, then $\nu(B_n)\leq \nu(B)=\nu([B\setminus B_n] \cup B_n)\leq \nu(B\setminus B_n)+\nu(B_n)$. Since $B\setminus B_n\downarrow \emptyset$, the result follows immediately.

From Eberlein-\v Smulian Theorem, we can use sequences to describe continuity. Let $(x_n^*)_{n\in \IN}$, be a sequence in $X^*$ converging weakly* to $x^*$. Hence, $x^*_n f$ converges to $x^* f$ pointwise. Set for every $n\in \IN,$
$$u_n=\underset{k\geq n}{\sup}\; x_k^*f\text{ and }l_n=\underset{k\geq n}{\inf}\; x_k^*f$$

Then, $(u_n)_{n\in \IN}$ and $(l_n)_{n\in \IN}$ are measurable, $(u_n)_{n\in \IN}$ is decreasing and converges pointwise to $x^*f$ and   $(l_n)_{n\in \IN}$ is increasing and converges pointwise to $x^*f$. From condition \eqref{DCB} both $(u_n)_{n\in \IN}$ and $(l_n)_{n\in \IN}$ are $\nu$-integrable on $E$, because for every $\omega\in E$ and every $n\in \IN$, $|u_n(\omega)|\leq \underset{k\in \IN}{\sup}|x^*_k f(\omega)|\leq \underset{k\in \IN}{\sup}\|x^*_k\|_{X^*} \|f(\omega)\|_X$ and since $x^*_n$ is weakly* convergent, it is bounded, then there is $M>0$ such that $|u_n(\omega)|\leq  M\|f(\omega)\|_X$, for every $\omega\in \Omega$. Hence, by the monotonicity of $\nu$ and its sub-additivity, $|\int_\Omega u_n \chi_E d\nu |\leq M \int_\Omega \|f\|_X \chi_E d\nu$ (\citet{DEN94}, corollary 6.6, p. 82). Note that this argument applies as well for $l_n,n\in \IN$. Let $\varepsilon>0$ and set for every $t\in \IR$, $A_t=\cap_{n\in \IN}\{\omega\in \Omega: u_n(\omega)\chi_E(\omega)>t\}$. Remark that, $\{\omega\in \Omega: x^* f(\omega)\chi_E(\omega)>t\}\subset A_t\subset \{\omega\in \Omega: x^* f(\omega)\chi_E(\omega)+\varepsilon>t\}$. In another hand, $\{\omega\in \Omega: l_n(\omega)\chi_E(\omega)>t\}\uparrow \{\omega\in \Omega: x^* f(\omega)\chi_E(\omega)>t\}$. This provides  $t\mapsto \nu(\{\omega\in \Omega: u_n(\omega)\chi_E(\omega)>t\})$ converges pointwise to $t\mapsto \nu(A_t)$ and $t\mapsto \nu(\{\omega\in \Omega: l_n(\omega)\chi_E(\omega)>t\})$ converges pointwise to $\nu\{\omega\in \Omega: x^* f(\omega)\chi_E(\omega)>t\}$. Also, $t\rightarrow \nu(A_t)$ is measurable as it is decreasing on $\IR$ and for every $t\in \IR$, $\nu(\{\omega\in \Omega: x^* f(\omega)\chi_E(\omega)>t\})\leq \nu( A_t) \leq \nu(\{\omega\in \Omega: x^* f(\omega)\chi_E(\omega)+\varepsilon>t\})$.  

Considering the monotonicity of $\nu$ again~: 
\begin{equation}\begin{array}{rl} 
\limsup \int_\Omega x_n^* f \chi_E d\nu &\leq \lim\int_\Omega u_n \chi_E d\nu\\
&\leq \lim \left(\int_0^{+\infty} \nu\{\omega\in \Omega :u_n(\omega) \chi_E(\omega)>t\} d\lambda (t)\right.\\
&\left.\;\; +\int^0_{-\infty}\left[ \nu(\{\omega\in \Omega : u_n(\omega)\chi_E(\omega)>t\}) -\nu(\Omega)\right]d\lambda(t)\right)\\
&\leq \int_{0}^{+\infty}\nu(\{\omega\in \Omega: x^* f(\omega)\chi_E(\omega)+\varepsilon>t\})d\lambda(t)\\
&\;\;+\int^0_{-\infty}\left[ \nu(\{\omega\in \Omega : x^*f(\omega)\chi_E(\omega)+\varepsilon>t\}) -\nu(\Omega)\right]d\lambda(t) \\
&\leq \int_\Omega x^*f \chi_E+\varepsilon d\nu.\\
&\leq \int_\Omega x^*f \chi_E d\nu +\varepsilon \nu(\Omega). 
\end{array}
\end{equation}

Where, the last inequality uses the commonotonic additivity of the integral with respect to a capacity. Since $\varepsilon$ is arbitrary, we conclude that 

\begin{equation}\label{LS}
\limsup \int_\Omega x_n^* f\chi_E d\nu  \leq \int_\Omega x^*f \chi_E d\nu
\end{equation}

In another hand,  
\begin{equation}\label{LI}
\begin{array}{rl}
\liminf \int_\Omega x_n^* f \chi_E d\nu &\geq \lim\int_\Omega l_n\chi_E d\nu \\
&\geq \lim \left(\int_0^{+\infty} \nu\{\omega\in \Omega :l_n(\omega)\chi_E(\omega)>t\} d\lambda (t)\right.\\
&\left.\;\; +\int^0_{-\infty}\left[ \nu(\{\omega\in \Omega : l_n(\omega)\chi_E(\omega)>t\}) -\nu(\Omega)\right]d\lambda(t)\right)\\
&\geq \int_0^{+\infty}\nu(\{\omega\in \Omega: x^* f(\omega)\chi_E(\omega)>t\})d\nu\\
&\;\;+\int^0_{-\infty}\left[ \nu(\{\omega\in \Omega : x^*f(\omega)\chi_E(\omega)>t\}) -\nu(\Omega)\right]d\lambda(t) \\
&\geq \int_\Omega x^*f\chi_E d\nu.

\end{array}
\end{equation}

The weak* continuity of $I_f(E)$ results from \eqref{LS} and \eqref{LI}.

\end{proof}

\begin{remark}
 The proof of the previous theorem can be performed using the dominated convergence theorem for capacities (\citet{DEN94}, Proposition 8.5, Proposition 8.8, and Theorem 8.9, p. 98-101). We provided above an alternative explicit proof. 
\end{remark}
\subsection{The indefinite Dunford-Choquet capacity}
Consider a Dunford-Choquet integrable function $f:\Omega\rightarrow X$ and denote by $I_f(\cdot)$ the indefinite Dunford-Choquet integral of $f$, that is for every $E$, $I_f(E) : X^*\rightarrow \IR$, defined at every $x^*\in X^*$ by $$I_f(E)x^*=\int_\Omega x^*f \chi_E d\nu.$$
 
Let us denote by $\HH_0$ the set of all positively homogeneous real valued functions defined on $X^*$. Endowed with the usual operations on spaces of real valued functions, addition of functions and scalar multiplication, $\HH_0$ is a real vector space. 

In the sequel $I_f$ is seen as a capacity from $\AA$ to $\HH_0$, even if in the absence of an ordering on $\HH_0$, some the properties of a capacity listed above are not meaningful.

A set $E\in \AA$ is said to be $\nu$-null if $\nu(Z\cap E)=0$, for every $Z\in \AA$. The $I_f$-null sets are defined similarly. The capacity $\nu$ is said to be null-additive iff for every $E,Z\in \AA$, such that $Z$ is $\nu$-null, $\nu(E\cup Z)=\nu(E)$. The $I_f$ null-additivity is defined similarly. 

\begin{proposition} The following holds~:
\begin{itemize}
\item[1)]If $E\in \AA$ is $\nu$-null, then $E$ is $I_f$-null. 
\item[2)]If $E\in \AA$ is $I_f$-null and $0\notin \cob(f(E))$, then $E$ is $\nu$-null.
\item[3)]If $\nu$ is null-additive, then $I_f$ is null-additive. 

\end{itemize}
\end{proposition}

\begin{proof}
1) Let $E\in \AA$ be a $\nu$-null set. Then, for every $Z\in \AA$, since $\nu(Z\cap E)=0$,  $\int_\Omega x^*f \chi_{Z\cap E} d\nu =0$, for every $x^*\in X^*$. That is, for every $Z\in \AA$, $I_f(Z\cap E)=0$. Hence, $E$ is $I_f$-null.

2) Let $E\in \AA$ be an $I_f$-null set such that $0\notin \cob(f(E))$. Then, for every $Z\in \AA$, for every $x^*\in X^*$, $\int_\Omega x^* f \chi_{Z\cap E} d\nu=0$. Let $Z\in \AA$ be fixed. Assume, \emph{per absurdum}, that $\nu(Z\cap E)>0$ and let us separate strictly according to Hahn-Banach separation theorem $\cob (f(Z\cap E))$ and $\{0\}$. I.e., take $x^*\in X^*$ and $\alpha>0$ such that, $x^*(y)>\alpha>x^*(0)=0$ for every $y\in \cob(f(Z\cap E))$. Then,  $x^*f(\omega)>\alpha>x^*(0)=0$ for every $\omega\in Z\cap E$. Then, 

$$
\begin{array}{rl}
\int_\Omega x^* f \chi_{Z\cap E} d\nu& =\int_{\IR^+} \nu(\{\omega \in \Omega : x^*f(\omega)\chi_{Z\cap E}(\omega)>t\}) d\lambda(t)\\
&>\int_{[0,\; \alpha]} \nu(Z\cap E)d\lambda=\alpha\nu(Z\cap E)>0.
\end{array}
$$ 
This contradicts the fact that $E$ is $I_f$-null and achieves the proof of 2).

3)  Let $E,Z\in \AA$ such that $Z$ is $\nu$-null. Let $x^*\in X^*$. Since $\nu$ is null-additive

$$
\begin{array}{rl}
I_f(E\cup Z)x^*=&\int_\Omega x^*f \chi_{E\cup Z} d\nu\\
=&\int_{\IR_+} \nu(\{\omega \in \Omega : x^*f (\omega) \chi_{E\cup Z}(\omega)>t\})d\lambda(t)\\
& +\int_{\IR_-} [\nu(\{\omega \in \Omega : x^*f (\omega)\chi_{E\cup Z}(\omega)>t\}-\nu(\Omega))]d\lambda(t)\\
=&\int_{\IR_+} \nu(\{\omega \in \Omega  : x^*f (\omega)\chi_{E}(\omega)>t\})d\lambda(t) \\
&+ \int_{\IR_-} [\nu(\{\omega \in \Omega : x^*f (\omega)\chi_{E}(\omega)>t\}-\nu(\Omega))]d\lambda(t)\\
=&\int_\Omega x^*f \chi_E d\nu\\
=&I_f(E)x^*.
\end{array}
$$
\end{proof}

 It is obvious that in point 2), the condition $0\notin \cob(f(E))$ cannot be relaxed. To see this,  remark that the indefinite  Dunford-Choquet integral of the null function $f\equiv 0$, relatively to an arbitrary finite and monotonic capacity, defines a null capacity  $I_0\equiv 0$. Then, every $E\in \AA$, is $I_0$-null, whatever is 	the value $\nu(E)$.

\begin{definition}
We say that ${I}_f$ is weakly sub-additive (resp. super-additive) iff for every $x^*\in X^*$ and $A,B\in \AA$, such that $A\cap B=\emptyset$, $  I_f (A\cup B) x^* \leq   I_f (A)x^*+  I_f (B) x^*$  (resp. $  I_f (A\cup B)x^*\geq   I_f (A)x^*+  I_f (B) x^*$).
\end{definition}

\begin{proposition}\label{PWA}$\;$
\begin{itemize} 
\item[1)]If $\nu$ is  sub-modular, then, $  I_f$ is weakly sub-additive.
\item[2)]If $\nu$ is super-modular, then, $  I_f$ is weakly super-additive. 
\end{itemize}
\end{proposition}

\begin{proof}Let $A,B\in \AA$, such that $A\cap B=\emptyset$.

1)$\displaystyle  I_f(A\cup B)x^*=\int_\Omega x^* f \chi_{A\cup B} d\nu=\int_\Omega x^* f\chi_A+x^* f\chi_B d\nu$, and from the sub-additivity theorem (\citet{DEN94}, theorem 6.3, p. 75), 

$\displaystyle \int_\Omega x^* f\chi_A+x^* f\chi_B d\nu\leq \int_\Omega x^* f\chi_{A} d\nu+ \int_\Omega x^* f\chi_B d\nu=  I_f(A)(x^*)+  I_f(B)x^*$. 


2) The proof is analogous to 1) using the super-additivity theorem (\citet{DEN94}, corollary 6.4, p. 78). 

\end{proof}

Denote by $\HH$ the subspace of $\HH_0$ consisting of all positively homogeneous functions defined on $X^*$ that are bounded on $B_{X^*}$. Since $B_{X^*}$ is compact for the weak* topology, $\HH$ includes the weak* continuous positively homogenous functions defined on $X^*$. It is easy to see that the following function defines a norm on $\HH$~: 

$$ \|I\|_\HH=\underset{x^*\in B_{X^*} }{\sup} |I(x^*)|.$$ 


Consider a function $f:\Omega\rightarrow X$ that is Dunford-Choquet integrable satisfying condition \eqref{DCB} of Theorem \ref{DCBT} on every $E\in \AA$. Following Theorem \ref{DCBT}, if $\nu$ is sub-additive and continuous from above, then for every $E\in \AA$, $I_f(E)\in \HH$. The following theorem shows that $ I_f(\cdot)$ defines a capacity continuous for $\|\cdot \|_\HH$

\begin{theorem}\label{TCON} Assume that $\nu$ is sub-additive and continuous from above and let  $f:\Omega\rightarrow X$ be a Dunford-Choquet integrable function satisfying condition \eqref{DCB} of Theorem \ref{DCBT} on every $E\in \AA$. Then, the capacity $I_f$ is continuous for the norm $\|\cdot\|_\HH$.
\end{theorem}


\begin{proof} In the sequel, given a function $h:\Omega\rightarrow \IR$, we use the following notations $h^+=\max\{ h, 0\}$ and $h^-=-\min\{ h, 0\}$.  
Let $(A_n)_{n\in \IN}$ be a sequence in $\AA$ and $A\in \AA$, such that $A_n\downarrow A$. Then, for every $x^*\in X^*$, $(x^*f)^+\chi_{A_n} $ decreases pointwise to $(x^*f)^+\chi_A$.  Since all these functions are bounded by $\nu$-integrable functions, for every $x^*\in X^*$, $\int_\Omega (x^*f)^+\chi_{A_n} d\nu$ and $\int_\Omega (x^*f)^+\chi_A d\nu$ exist and, using the continuity from above of $\nu$ and its sub-additivity, we deduce that $\int_\Omega (x^*f)^+\chi_{A_n} d\nu$ decreases and converges to $\int_\Omega (x^*f)^+\chi_A d\nu$ (\citet{DEN94}, Proposition 8.5, Proposition 8.8, and Theorem 8.9, p. 98-101). 
In another hand, it is easy to see that all the functions constituting this sequence as well as the limite $x^*\mapsto\int_\Omega (x^*f)^+\chi_A d\nu$ are continuous for the weak* topology. Indeed, if $(x^*_n)_{n\in \IN}$ is a  sequence in $X^*$ converging weak* to $x^*$, then $(x_n^*f)^+$ converges pointwise to $(x^*f)^+$. Since all these functions are bounded by $\nu$-integrable functions, for every $B\in \AA$, $\int_\Omega (x_n^*f)^+\chi_B d\nu$ converges to $\int_\Omega (x^*f)^+\chi_B d\nu$ and this provides the needed weak* continuity.   

Now, since $B_{X^*}$ is weak* compact, Dini's theorem guarantees that the convergence of the sequence  of functions $x^*\mapsto \int_\Omega (x^*f)^+\chi_{A_n} d\nu,n\in \IN,$ to $x^*\mapsto \int_\Omega (x^*f)^+\chi_A d\nu$  is uniform on $B_{X^*}$. 
Similar arguments proves as well that the sequence of functions $x^*\mapsto \int_\Omega - (x^*f)^-\chi_{A_n} d\nu$ converges uniformly to $x^*\mapsto \int_\Omega - (x^*f)^-\chi_{A} d\nu$  on $B_{X^*}$.  As a result, using the additive comonotonicity property of the Choquet integral, 

$\int_\Omega x^*f\chi_{A_n} d\nu=\int_\Omega (x^*f)^+\chi_{A_n} d\nu+\int_\Omega - (x^*f)^-\chi_{A_n}d\nu$ converges uniformly to $\int_\Omega (x^*f)^+\chi_{A} d\nu+\int_\Omega - (x^*f)^-\chi_{A} d\nu=\int_\Omega x^*f\chi_A d\nu$ on $B_{X^*}$. It follow that 

$$\underset{x^*\in B_{X^*} }{\sup} |I_f(A_n)x^*- I_f(A)x^*|\rightarrow 0.$$
I.e.,

$$ \| I_f(A_n)- I_f(A)\|_\HH\rightarrow 0.$$
That is, $ I_f$ is continuous from above. The continuity from below of $ I_f$ is analogous, then not reproduced. It result that $ I_f$ is continuous.   
\end{proof}


\section*{Comment}  
\begin{itemize}

\item[1)] It is worth noting that for the Choquet integral, given $E\in \AA$ and a Choquet integrable function $h:\Omega\rightarrow \IR$, $\int_\Omega h \chi_E d\nu$ is not necessarily equal to $\int_E h d\nu$. Then, if we replace in the definition of $I_f(E)$ above $\int_\Omega x^* f \chi_Ed\nu $ by $\int_E x^* fd\nu$, the proofs of Proposition \ref{PWA} and of Theorem \ref{TCON} will not work. However, one can remedy this situation for Theorem \ref{TCON}, by the use of the symmetric version of the Choquet integral,  the commonotonic additivity argument in the proof, will follow from the definition of the symmetric integral. 
\item[2)]  The results in \citep{DEN94} we used above, like the dominated convergence theorems, are stated for monotonic capacities defined on the set $\PP(\Omega)$ of all subsets of $\Omega$. It is clear that theses results are still valid in our setting, since every monotonic capacity on $\AA\subset \PP(\Omega)$ extends, by the Zorn's lemma to a monotonic capacity on $\PP(\Omega)$ and obviously the corresponding integrals of $\AA$-measurable functions coincides. To see how to obtain such extensions, consider a monotonic capacity $\nu$ on $\AA$ and set $\CC=\{(\mu,\Sigma) : \AA\subset \Sigma, \mu\text{ monotonic on }\Sigma\text{ and } \mu_{|\AA}\equiv \nu\}$.  Consider the following ordering on $\CC$ : $$(\mu_1,\Sigma_1)\leq (\mu_2,\Sigma_2) \Leftrightarrow  \Sigma_1\subset \Sigma_2 \text{ and } {\mu_2}_{|\Sigma_1}\equiv \mu_1.$$
This makes $\CC$ a directed set. Consider a chain $(\mu_\alpha,\Sigma_\alpha)$ in $\CC$ and set $\Sigma=\cup_\alpha \Sigma_\alpha$ and define $\mu$ on $\Sigma$ by $\mu_{|\Sigma_\alpha}=\mu_\alpha$. Obviously $(\mu,\Sigma)\in \CC$. Hence, every chain in $\CC$ has an upper bound in $\CC$. Zorn's Lemma guarantees that $\CC$ has a maximal element, denote it again $(\mu,\Sigma)$. Necessarily, $\Sigma=\PP(A)$, otherwise, let $A\in \PP(\Omega)\setminus \Sigma$ and set $\bar \mu(A)=\inf\{\mu (E) : A\subset E\text{ and }E\in \Sigma\}$ and $\bar \mu\equiv \mu$ on $\Sigma$. Since, $\Omega\in \AA\subset \Sigma$, $\bar \mu(A)$ is well defined. It is easy to see that $(\bar \mu,\Sigma\cup\{A\})\in \CC$ and then $(\bar \mu,\Sigma\cup\{A\})$ is strictly greater than $(\mu,\Sigma)$ contradicting the maximality of the last element and proves the existence of the needed extension.

\end{itemize}

\nocite{}

\end{document}